\documentclass[12pt]{article}

\usepackage{amsthm}
\usepackage{fullpage}

\usepackage{amssymb, amsmath}
\usepackage{hyperref}
\usepackage{graphicx}

\newtheorem{thm}{Theorem}
\newtheorem{prop}[thm]{Proposition}
\newtheorem{lem}[thm]{Lemma}
\newtheorem{cor}[thm]{Corollary}

\theoremstyle{definition}

\theoremstyle{remark}

\numberwithin{equation}{section}

\renewcommand{\tau}{\uptau}
\renewcommand{\phi}{\varphi}
\renewcommand{\epsilon}{\varepsilon}

\author{
Fabr\'{i}cio S. Benevides
\thanks{Departamento de Matem\'{a}tica, Universidade Federal do Cear\'{a}, Fortaleza, CE 60455-760, Brazil {\tt fabricio@mat.ufc.br}. Research supported by CNPq.}
\and
D\'{a}niel Gerbner
\thanks{Hungarian Academy of Sciences, Alfr\'{e}d R\'{e}nyi Institute of Mathematics, P.O.B. 127, Budapest H-1364, Hungary {\tt gerbner.daniel@renyi.mta.hu}. Research supported      by Hungarian National Science Fund (OTKA), grant PD 109537.}
\and
Cory T. Palmer
\thanks{Department of Mathematical Sciences, University of Montana, Missoula, Montana 59812, USA {\tt cory.palmer@umontana.edu}. Research supported by Hungarian National Science Fund (OTKA), grant NK 78439.}
\and
Dominik K. Vu
\thanks{Department of Mathematical Sciences, University of Memphis, Memphis, Tennessee 38152, USA {\tt dominik.vu@memphis.edu}. Research supported in part by the National Science Foundation, grants  DMS-0906634 and CNS-0721983, and by the Heilbronn Fund.}
}
\title{Generalising separating families of fixed size}

\begin{document}

\maketitle

\begin{abstract}
We examine the following version of a classic combinatorial search problem introduced by R\'{e}nyi:
Given a finite set $X$ of $n$ elements we want to identify an unknown subset $Y \subset X$ of exactly $d$ elements by testing, by as few as possible subsets $A$ of $X$, whether $A$ contains an element of $Y$ or not. We are primarily concerned with the model where the family of test sets is specified in advance (non-adaptive) and each test set is of size at most a given $k$. Our main results are asymptotically sharp bounds on the minimum number of tests necessary for fixed $d$ and $k$ and for $n$ tending to infinity.
\end{abstract}

\section{Introduction} 

We consider a central question of combinatorial search theory in which, given a set $X$, we wish to identify a particular subset $Y$ of unknown elements of $X$. We call the elements of $Y$ \emph{defective}. To this end, we are allowed to construct a family $\mathcal{A}$ of \emph{queries}. Each query in $\mathcal{A}$ corresponds to a subset $A \subset X$ and we receive a positive result if and only if $A$ contains at least one of the elements of $Y$. Here we are concerned with the so called non-adaptive case, in which the queries are chosen in advance, so we cannot modify $\mathcal{A}$ based on the answers to some of the queries. The typical goal is to find the minimum size of a family $\mathcal{A}$ that is required to determine any set $Y$. This question is related to many practical problems, amongst which are \emph{Wasserman-type blood tests}, chemical analysis and the defective coin problem~\cite{Dorfmann1943, Sterrett1957}. A comprehensive overview of the main types of combinatorial search problems can be found in a survey by Katona~\cite{Katona1973} or the monograph of Du and Hwang~\cite{DuHwang2000}. 

In order to identify a fixed defective set of size at most (or exactly) $d$ among $n$ elements it is well known (see e.g. \cite{DuHwang2000}) that the number of queries required, denoted $q(n,d)$, satisfies
\[\Omega\left(\frac{d^2}{\log d} \log n \right) \leq q(n,d) \leq O(d^2 \log n).\]

In this paper, we restrict our attention to the case where the query sets may only be of size at most $k$. For this model, the particular case where $Y$ contains a single element was posed as a problem by R\' enyi \cite{Renyi1961a} and solved by Katona \cite{Katona1966} for $k < n/2$. Katona determined the exact form of a matrix representing an optimal search and used this to find upper and lower estimates for the minimum number of queries. While the lower bound provided is best known, the upper bound was subsequently improved by Wegener~\cite{Wegener1979} and Luzgin~\cite{Luzgin1980}. In 2008 Ahlswede~\cite{Ahlswede2008} proved that the lower bound is asymptotically tight.

Let $X$ be a set of size $n$ and $Y\subset X$ be a set of defective elements of size at most $d$. Let $q(n,\overline{d},k)$ denote the least number of queries of size at most  $k$ necessary to identify $Y$.
In 2013, Hosszu, Tapolcai and Wiener~\cite{HosszuTapolcaiWiener2013}  strengthened Katona's result while providing a proof entirely relying on linear algebraic methods.

\begin{thm}[Hosszu, Tapolcai and Wiener]\label{thm:HTW}
For $k <n/2$, $q(n,\overline{1},k)$ is the least number $q$ for which there exist positive integers $j \leq q-1$ and $a < \binom{q}{j+1}$ such that
\begin{align*}
\sum_{i=0}^{j} i \cdot \binom{q}{i} + a (j+1)  &\leq kq,\\
\sum_{i=0}^{j} \binom{q}{i} + a &= n.
\end{align*}
\end{thm}

When $n$ is large enough this gives the following corollary.

\begin{cor}[Hosszu, Tapolcai and Wiener]
If $n \geq \binom{k}{2} +1$, then \[q(n,\overline{1},k) = \left\lceil  \frac{2n-2}{k-1} \right\rceil.\]
\end{cor}

When the defective set is of size at most $d$, where $d>1$, 
D'yachkov and Rykov \cite{DyachkovRykov2002} proved a general lower bound and found conditions for when this lower bound is sharp (see also F\" uredi and Ruszink\'o~\cite{FurediRusziko2013}).

\begin{thm}[D'yachkov and Rykov]\label{lower-bound} If $n \geq k \geq d \geq 2$, then
\[ \left \lceil \frac{dn}{k} \right\rceil \leq q(n,\overline{d},k).\]
Furthermore, if $d \geq 3$, $k \geq d+1$ and $n = k^d$, then
\[q(n,\overline{d},k) = \frac{dn}{k} = dk^{d-1}.\]
\end{thm}

In light of the above results, we focus on the case when the defective set $Y$ has size exactly $d$. This allows for a smaller number of queries to determine $Y$. Define $q(n,d,k)$ as the minimum number of queries needed to find a fixed defective set $Y$ of size exactly $d$ among $n$ elements. Our main theorem gives bounds on $q(n,d,k)$ that are asymptotically sharp when $d$ is even.

\begin{thm}\label{thm:UBNAdaptive} Fix an integer $d\geq 2$.
If $n\geq k\geq \lfloor d/2 \rfloor + 1$, then
\[\left \lceil \frac{(\lfloor d/2 \rfloor +1)n}{k} \right \rceil -1 \leq q(n, d, k).\]
Furthermore, if $k\geq 2$ and $n$ is sufficiently large, then
\[q(n,d,k) \leq  \left\lceil \frac{(\lceil d/2\rceil +1)(n-1)}{k}\right\rceil + (\lceil d/2\rceil+1).\]
\end{thm}

Note that by querying singleton sets we can identify any defective set of size exactly $d$ using $n-1$ queries (or a defective set of any size using $n$ queries). So we have the trivial upper bound $q(n,d,k) \leq n-1$. In particular, this means that the upper bound above is only of interest when $k> \lceil d/2 \rceil +1$.

When $d=2,3$ these bounds can be improved.

\begin{thm}\label{thm:UBtightd23}
Let $n\geq k$ be positive integers with $n$ sufficiently large, then
\begin{enumerate}
\item[(a)] \[q(n,2,k) = \left\lceil\frac{2(n-1)}{k} \right\rceil\]
\item[(b)] \[ \left\lceil \frac{3n-2}{k+3} \right\rceil \leq q(n,3,k) \leq \left\lceil  \frac{3n}{k} \right\rceil+2.\]
\end{enumerate}
\end{thm}

We use the notation $[n]$ to denote the set $\{1, 2, \ldots, n\}$. A family $\mathcal{A}$ of subsets  of $[n]$ is called $d$-\emph{separating} ($\overline{d}$-\emph{separating}) if for any two distinct sets $D_1,D_2\subset [n]$ of size $d$ (at most $d$, respectively) we have a member $A \in \mathcal{A}$ such that either
\[A \cap D_1 \not = \emptyset \textrm{ and } A \cap D_2 = \emptyset, \]
or
\[A \cap D_1 = \emptyset \textrm{ and } A \cap D_2 \not = \emptyset.\]

It is well known that for a fixed defective set $Y$ of size $d$ a family of queries can determine $Y$ if and only if $\mathcal{A}$ is a $d$-separating family. The separating property is monotone in the following sense: a $d$-separating family is $\ell$-separating for any $\ell \leq d$.

For $|\mathcal{A}| = q$, we can form a $q \times n$ matrix $M$  such that the rows of $M$ are the characteristic vectors of the members of $\mathcal{A}$. The columns of $M$ can be thought of as characteristic vectors of a hypergraph $\mathcal{H}$ on the vertex set $[q]$. It is easy to see that $\mathcal{A}$ is $d$-separating if and only if $\mathcal{H}$ is $d$-\emph{union-free}, that is, every collection of exactly $d$ distinct members of $\mathcal{H}$ has a unique union. For a given family of queries $\mathcal{A}$ we will call such a hypergraph $\mathcal{H}$ the \emph{dual hypergraph} of $\mathcal{A}$. Note that in this model it is possible for $\mathcal{H}$ to have the empty set as a hyperedge\footnote{Note that the property that $\mathcal{A}$ is $d$-separating, prevents the dual hypergraph from having multi-hyperedges.}.

Clearly if a family is $\overline{d}$-separating, then it is also $d$-separating. Chen and Hwang~\cite{ChenHwang2007} provide a relationship in the other direction.

\begin{thm}[Chen and Hwang]\label{chen-hwang}
If $\mathcal{A}$ is $2d$-separating family, then there exists a $\overline{d+1}$-separating family $\mathcal{A}'$ obtained by adding at most one new element to the ground set of $\mathcal{A}$.
\end{thm}

This theorem is the primary tool for the lower bound of Theorem~\ref{thm:UBNAdaptive}. Chen and Hwang state that their theorem is weak in the sense that it should be possible to construct a $\overline{\ell}$-separating family $\mathcal{A}'$ for $\ell > d+1$. However, it follows from our upper bound in Theorem~\ref{thm:UBNAdaptive} that in general $\ell$ cannot be improved to $d+2$.

\section{General bounds on $q(n,d,k)$}

We begin by proving the lower bound in Theorem~\ref{thm:UBNAdaptive}. Suppose that $\mathcal{A}$ is a $d$-separating family on ground set $[n]$ with query size at most $k$ such that
\[|\mathcal{A}| < \left \lceil \frac{(\lfloor d/2 \rfloor +1)n}{k} \right \rceil -1.\]
Set $\ell = \lceil d/2 \rceil+1$. By Theorem~\ref{chen-hwang} we can add at most one new element to the ground set of $\mathcal{A}$ in order to obtain an $\overline{\ell}$-separating family $\mathcal{A}'$ such that
\[|\mathcal{A}'| < \left \lceil \frac{\ell n}{k} \right \rceil.\]
This contradicts the lower bound given by Theorem~\ref{lower-bound}.

To prove the upper bound in Theorem~\ref{thm:UBNAdaptive} we show an explicit construction of the dual hypergraph. Recall that a hypergraph $\mathcal{H}$ is $\ell$-\emph{uniform} if all hyperedges are of size $\ell$. Furthermore a hypergraph $\mathcal{H}$ is \emph{linear} if every two hyperedges intersect in at most one vertex.

A hypergraph is said to be a \emph{cycle} if it has at least two edges and there exists a cyclic ordering of its edges $\{e_1, \dots, e_\ell\}$ such that there are distinct vertices $v_1, \dots, v_\ell $ such that $v_i = e_i \cap e_{i+1}$ (where $e_{\ell+1} = e_1$). This concept of a cycle in a hypergraph is sometimes called \emph{Berge-cycle}, after C. Berge~\cite{Berge1989}. The \emph{length} of a cycle is the number of edges it contains and the \emph{girth} of a hypergraph is the length of the shortest cycle it contains. We use the term \emph{triangle} and $C_4$ to refer to hypergraph cycles with three and four hyperedges, respectively.

We begin with a lemma relating the uniformity and the property of being union-free for hypergraphs of girth at least $5$.

\begin{lem}\label{girth-sep}
Let $\ell \geq 2$ and $G$ be an $\ell$-uniform linear hypergraph. If $G$ has girth at least $5$, then $G$ is $(2\ell-2)$-union-free.  
\end{lem}

\begin{proof}
Suppose $G$ is not $(2 \ell -2)$-union-free, then there exist two distinct collections of edges 
$\mathcal{D} = \{D_1, \dots, D_{2\ell-2}\}$ and 
$\mathcal{E} = \{E_1, \dots, E_{2\ell-2}\}$
such that
\[\bigcup_{i=1}^{2\ell-2} D_i = \bigcup_{i=1}^{2\ell-2} E_i.\]

Consider the case when there are two sets $D_1,D_2$ that are both not members of $\mathcal{E}$. If $D_1$ and $D_2$ are not disjoint, then their union has $2\ell-1$ elements. These elements are covered by the union of the $E_i$s and each $E_i$ contains at most one element from each of the sets $D_1$ and $D_2$. Thus there must be an $E_i$ that intersects both $D_1$ and $D_2$ in distinct vertex. Therefore, $E_i, D_1,D_2$ form a triangle; a contradiction. On the other hand, if $D_1$ and $D_2$ are disjoint, then their union has $2\ell$ elements. These elements are covered by the union of the $E_i$s and each $E_i$ contains at most one element from each of the sets $D_1$ and $D_2$. Thus there must be an $E_i$ and $E_j$ that intersect both $D_1$ and $D_2$ in four distinct vertices. Therefore, $E_i, D_1, E_j,D_2$ form a $C_4$; a contradiction.

Consider the case when there is exactly one set in $\mathcal{D}$ (say $D_1$) that is not a member of $\mathcal{E}$. Consequentially, there is a set in $\mathcal{E}$ (say $E_1$) that is not in $\mathcal{D}$. The remaining  $2\ell-3$ members of $\mathcal{D}$ and $\mathcal{E}$ are the same. The union of $D_1$ and $E_1$ has at least $2\ell-1$ elements. All the vertices of the union except for the intersection must be covered by the remaining $2\ell - 3$ edges of $\mathcal{D}$. As in the previous case we get either a triangle or a $C_4$ (depending on whether $D_1$ and $E_1$ intersect); a contradiction.
\end{proof}

Ellis and Linial~\cite{EllisLinial2014} (using a result of Cooper, Frieze, Molloy, and Reed~\cite{CooperFriezeMolloyReed1996}) constructed a regular uniform hypergraph with girth at least $5$.

\begin{thm}[Ellis and Linial~\cite{EllisLinial2014}]\label{ellis-linial}
Fix integers $\ell \geq 3$ and $k \geq 2$. Then for every $m$ large enough such that $\ell$ divides $m$, there exists a linear $k$-regular $\ell$-uniform hypergraph on $m$ vertices with girth at least $5$.
\end{thm}

We now construct a $d$-union-free hypergraph on at most $\left \lceil \frac{(\lceil d/2\rceil +1)(n-1)}{k}\right\rceil + (\lceil d/2\rceil+1)$ vertices and with at least $n$ hyperedges. This will be the dual hypergraph of a $d$-separating family of sets which gives the upper bound in Theorem~\ref{thm:UBNAdaptive}.

Set $\ell = \lfloor d/2 \rfloor +1$ and let $q$ be the smallest integer such that $\lceil \frac{\ell n}{k} \rceil \leq q$ and $q$ is divisible by $\ell$.
Thus
\[\left \lceil \frac{\ell (n-1)}{k} \right \rceil \leq q \leq \left \lceil \frac{\ell (n-1)}{k} \right \rceil + \ell = \left\lceil \frac{(\lceil d/2\rceil +1)(n-1)}{k}\right\rceil + (\lceil d/2\rceil+1).\]
Let $\mathcal{H}$ be a linear $k$-regular $\ell$-uniform hypergraph on $q$ vertices (by Theorem~\ref{ellis-linial}). The number of hyperedges in $\mathcal{H}$ is
\[\frac{kq}{\ell} \geq \frac{k}{\ell} \left \lceil \frac{\ell (n-1)}{k} \right \rceil \geq n-1.\]

By Lemma~\ref{girth-sep} $\mathcal{H}$ is $d$-union-free (in fact, when $d$ is odd $\mathcal{H}$ is $(d+1)$-union-free). Now let us add the empty set (as a hyperedge) to $\mathcal{H}$ to get a hypergraph with at least $n$ hyperedges. This new hypergraph is still $d$-union-free.

\section{Bounds for small defective sets}\label{sec:LBdNAdaptive}
In this section we prove Theorem~\ref{thm:UBtightd23}.

\subsection{Two defective elements -- Proof of Theorem~\ref{thm:UBtightd23}(a)}

Instead of applying the theorem of Ellis and Linial as above, we can use a version of the classic result of Erd\H{o}s and Sachs~\cite{ErdosSachs1963} on the existence of graphs of arbitrary girth. This allows for a concrete bound on the threshold for $n$.

\begin{thm}[Erd\H os and Sachs]\label{ErdosSachs}
Fix integers $k\geq 2$ and $g\geq 4$, and let $m \geq 4k^{g}$ be an even integer. Then there exists a $k$-regular graph on $m$ vertices with girth at least $g$.
\end{thm}

The following proposition gives the upper bound in Theorem~{\ref{thm:UBtightd23}(a)}.
\begin{prop} Fix an integer $k \geq 2$ and let $n > 2k^7$. Then
\[q(n, 2, k) \leq \left\lceil  \frac{2(n-1)}{k} \right\rceil.\]
\end{prop}

\begin{proof}
Let $q =  \left\lceil  \frac{2(n-1)}{k} \right\rceil \geq 4k^6$. We will construct a graph $G$ with girth at least $5$ on $q$ vertices with $n-1$ edges. By Lemma~\ref{girth-sep} we have that $G$ is a $2$-union-free (hyper)graph. Then we add the empty set (as a hyperedge) to $G$ to get a hypergraph on $q$ vertices with $n$ hyperedges. It is easy to see that if $G$ is $2$-union-free, then adding the empty set cannot destroy the $2$-union-free property.

We distinguish two cases based on the parity of $q$. First let us assume $q$ is even. By Theorem~\ref{ErdosSachs} there exists a $k$-regular graph $G$ on $q$ vertices with girth at least $5$. The number of edges in $G$ is $qk/2\geq n-1$. 

Now suppose $q$ is odd. By Theorem~\ref{ErdosSachs} there exists a $k$-regular graph $G'$ on $q+1$ vertices with girth at least $6$. Let us remove an arbitrary vertex $x$ from $G'$. Let $X$ be the neighborhood of $x$. The graph $G'$ is triangle-free, so $X$ is an independent set. Furthermore, $G'$ is $k$-regular, so $|X|=k$, so we can add a matching of size $\lfloor k/2\rfloor$ to the vertices of $X$. Let the resulting graph be $G$. It is easy to see that as $G'$ had girth at least $6$, the graph $G$ will have girth at least $5$. The number of edges in $G$ is at least
\[\frac{k(q+1)}{2} -k + \left \lfloor \frac{k}{2} \right \rfloor \geq n -1- \frac{k}{2} + \left\lfloor \frac{k}{2}\right \rfloor \geq n - 1- \frac{1}{2}.\]
Therefore, the number of edges in $G$ is at least $n-1$. 
\end{proof}

We now prove the lower bound on $q(n,2,k)$. Fix $k$ and $n$ and let $q$ be the minimal integer such that there exists a $2$-separating family with query size at most $k$. In the dual hypergraph $\mathcal{H}$, let $e\leq 1$ be the number of hyperedges of size $0$, let $s$ be the number of hyperedges of size $1$ that are contained in a hyperedge of size at least $3$, let $s'$ be the number of remaining hyperedges of size $1$ and let $t$ be the number of hyperedges of size at least $3$. Therefore, the number of hyperedges of size $2$ is $n-e-s-s'-t$. We need two simple lemmas relating these values.

\begin{lem}\label{size3}
Every hyperedge of size at least $3$ of $\mathcal{H}$ contains at most one hyperedge of size $1$. That is $s \leq t$.
\end{lem}

\begin{proof}
Assume otherwise that the hyperedge $h$ contains two hyperedges, say $\{a\}$ and $\{b\}$. Then $h \cup \{ a \} = h = h \cup \{b\}$ contradicting the $2$-union-free property of $\mathcal{H}$.
\end{proof}

\begin{lem}\label{deg1} 
If $\{a\}$ and $\{a,b\}$ are hyperedges of $\mathcal{H}$, then the degree of the vertex $b$ is $1$. 
\end{lem}

\begin{proof}
Assume there is an edge incident to $b$, called $h$, different from $\{a, b\}$. Then $h \cup \{a\}= h \cup \{ a,b \}$ contradicting the 2-union-free property of $\mathcal{H}$.
\end{proof}

Now let us count the number of pairs $(v,h)$ where $v$ is a vertex and $h$ is a hyperedge of $\mathcal{H}$ such that $v \in h$.

The maximum degree in $\mathcal{H}$ is $k$. Furthermore,
for each hyperedge $\{a\}$ of size $1$ not contained in a hyperedge of size at least $3$, there is a vertex of degree $1$ in $\mathcal{H}$. Indeed, either $\{a\}$ is isolated and thus $a$ is of degree $1$ or $\{a\}$ is in some hyperedge $\{a,b\}$ and by Lemma~\ref{deg1} we have that $b$ is of degree $1$. Thus we have at least $s'$ vertices of degree $1$. This implies that the number of pairs $(v,h)$ is at most
\[k(q-s') +s' = kq - (k-1)s'.\]
On the other hand by counting the size of all hyperedges we get the number of pairs $(v,h)$ is at least
\[s+s'+3t+2(n-e-s-s'-t).\]
Applying Lemma~\ref{size3} gives
\[s+s'+3t+2(n-e-s-s'-t) \geq 2s+s'+2t+2(n-e-s-s'-t) = 2n-2e+s'\geq 2n-2-s'.\]
Combining the upper and lower estimates for the number of pairs $(v,h)$ yields
\[2n-2-s' \leq kq - (k-1)s'.\]
Using the fact that $k \geq 2$ and solving for $q$ gives the lower bound.

\subsection{Three defective elements -- Proof of Theorem~\ref{thm:UBtightd23}(b)}\label{subsec}

The upper-bound follows from Theorem~\ref{thm:UBNAdaptive}.

For the lower bound fix $k$ and $n$ and let $q$ be the minimal integer such that there exists a $3$-separating family $\mathcal{A}$ with query size at most $k$. 
Let $\mathcal{H}$ be the dual hypergraph for $\mathcal{A}$. Note that this hypergraph is not necessarily uniform.

As in the case where $d=2$, we sum the sizes of all hyperedges in $\mathcal{H}$. There is at most $1$ hyperedge of size $0$ and at most $q$ many hyperedges of size $1$. First we show that there are not too many hyperedges of size $2$ in $\mathcal{H}$.

\begin{lem}\label{tree-structure}
The graph $G$ formed by the hyperedges of size $2$ in $\mathcal{H}$ is a forest.
\end{lem}
\begin{proof}
We show that $G$ contains no triangle, no $C_4$ and no path of length $4$. Such a graph is clearly a forest.

Recall that a $3$-separating family is also $2$-separating. Let $e,f,g$ be the edges of a triangle in $G$, then it is immediate that $e \cup f = e \cup g$ which violates the $2$-union-free property of $\mathcal{H}$. Similarly, if $e,f,g,h$ are the edges of a $C_4$ in $G$ (such that $e$ and $g$ are disjoint), then $e \cup g = f \cup h$ which violates the $2$-union-free property of $\mathcal{H}$. Finally, if $e,f,g,h$ are the edges of a path of length $4$ (in this order), then $e \cup f \cup h = e \cup g \cup h$ which violates the $3$-union-free property of $\mathcal{H}$. 
\end{proof}

Thus we have at most $q-1$ hyperedges of size $2$ in $\mathcal{H}$. Therefore, the sum of the sizes of the hyperedges of $\mathcal{H}$ is at least
\[q + 2(q-1) + 3(n-1-2q+1) = 3n - 3q -2 \]
The maximum degree in $\mathcal{H}$ is $k$ so the above sum is at most $qk$. Combining these two estimates and solving for $q$ yields
\[\frac{3n-2}{k+3} \leq q.\]

\section{Further results}\label{sec:ConclOPbs}

\subsection{Fixed query size}

Throughout the paper we have allowed queries to have size at most $k$. Katona~\cite{Katona1966} showed that when searching for a fixed defective set of size at most $1$ there is no difference in the minimum number of necessary queries whether we restrict the queries size to be at most $k$ or to be exactly $k$. Therefore it is somewhat unexpected that in the case of searching for a fixed defective set of size exactly $d$, for $d\geq 2$, we can have different answers depending on whether the query size is at most $k$ or exactly $k$.

To illustrate, let us examine the simplest case when $k=2$ and defective set is of size exactly $d$, for a given $d \geq 3$. We distinguish between two kind of restrictions: (i) each query set is of size at most $2$ or (ii) each query set is of size exactly $2$. By asking queries of size $1$ we can identify the defective set with $n-1$ queries, so $q(n,d,2) \leq n-1$. On the other hand, if a family of queries is $d$-separating, then it is $2$-separating and we can use Theorem~\ref{thm:UBtightd23} to get
\[q(n,d,2) \geq q(n,2,2) = \left \lceil \frac{2(n-1)}{2} \right \rceil = n-1.\]

Therefore, we have the following simple corollary.

\begin{cor}
If $n\geq d\geq 3$, then
\[q(n,d,2) = n-1.\]
\end{cor}

However, if we only allow queries of size exactly $2$ we cannot determine such a defective set with only $n-1$ queries.

\begin{prop}
Let $q$ be the minimum number such that there exists a family of queries of size exactly $2$ that can determine any defective set of size $d$, for a given $d\geq 3$. Then $q \geq n$.
\end{prop}

\begin{proof}
Let $\mathcal{H}$ be the dual hypergraph for the family of queries in the statement of the proposition. Then $\mathcal{H}$ is a $2$-regular hypergraph on $q$ vertices with $n$ hyperedges. If all hyperedges of $\mathcal{H}$ are of size $2$, then $\mathcal{H}$ is a forest (see Lemma~\ref{tree-structure} in Subsection~\ref{subsec}). Therefore, $\mathcal{H}$ has at least $n+1$ vertices, i.e. $q \geq n+1 \geq n$ and we are done.

Therefore, we may assume that $\mathcal{H}$ has at least one hyperedge of size other than $2$. In $\mathcal{H}$ let $e\leq 1$ be the number of hyperedges of size $0$, let $s$ be the number of hyperedges of size $1$, and let $t$ be the number of hyperedges of size at least $3$. Clearly as $\mathcal{H}$ is $2$-regular it cannot contain an isolated  hyperedge of size $1$. Furthermore, $\mathcal{H}$ cannot contain two hyperedges of the form $\{a\}$ and $\{a,b\}$ as $b$ must have degree $1$ in this case (see Lemma~\ref{deg1} in Subsection~\ref{subsec}). Therefore, every hyperedge of size $1$ is contained in a hyperedge of size at least $3$, i.e. $s \leq t$.

The sum of degrees in $\mathcal{H}$ is $2q$. By counting the sizes of all edges in $\mathcal{H}$ we obtain that
\[2q \geq s+3t+2(n-e-s-t).\]
If $e=0$, then using the fact that $s \leq t$ it follows that
\[2q \geq s+3t+2(n-s-t) \geq 2s + 2t + 2(n-s-t) = 2n\]
and we are done.

Now let us suppose that $e=1$, i.e. $\mathcal{H}$ contains the empty set as a hyperedge. 
In this case it is easy to see that $\mathcal{H}$ cannot contain any hyperedges of size $1$. Indeed if $\{a\}$ is a hyperedge, then there must be some $\{a,b,c\}$ hyperedge, but then 
\[\{a\} \cup \{a,b,c\} = \emptyset \cup \{a,b,c\}\]
violates the $d$-separating property of $\mathcal{H}$. Therefore, there must be at least one hyperedge of size at least $3$, i.e. $t\geq 1$. Thus,
\[2q \geq 3t+2(n-1-t) \geq 2n -2 + t \geq 2n-1.\]
Therefore $q \geq n$ and we are done.
\end{proof}

\subsection{Adaptive search}

We call the search model \emph{adaptive} if we ask the query sets in a sequence and allow that each query set $A$ may depend on the answer given for previous queries. 
As in the previous sections we are particularly interested in the case when the query sets are of size at most $k$. Let $Y$ be a defective set of at most $d$ elements.
The minimum number of queries required determine $Y$ among a set of size $n$ in the adaptive model is denoted by $t(n,\overline{d},k)$. In the case of $d=1$ the question was solved completely by Katona~\cite{Katona1973}.

\begin{thm}[Katona]\label{adaptive}
Let $n,k$ be integers, such that $k<n/2$, then 
\[t(n,\overline{1},k) = \left\lceil \frac{n}{k} \right\rceil - 2 + \left\lceil \log \left(n - k \left\lceil \frac{n}{k} \right\rceil + 2k \right) \right\rceil.\]
\end{thm}

The proof of Theorem~\ref{adaptive} can be easily generalized for larger defective sets.

\begin{thm}
For any integers $k, n>k, d>1$ it holds
\[\left \lceil \frac{n}{k} \right \rceil - 2 + \log \binom{k+1}{d} \leq t(n,\overline{d},k) \leq \left \lceil \frac{n}{k} \right\rceil -2 + d\left\lceil 1+\log k\right\rceil.\]
\end{thm}

\begin{proof}[Proof sketch.]
For the upper bound we simply ask disjoint query sets of size $k$. Whenever we get a positive answer we perform a standard binary search on the $k$ set to determine one of the defective elements. A second query is needed to determine if the $k-1$ remaining elements still contain a defective element. If so we can perform another binary search. If not we continue with another disjoint set of size $k$. When there are $2k$ elements remaining we can repeatedly perform a binary search to find the remaining defective elements.

For the lower bound suppose that we get negative answers for the first $\lceil \frac{n}{k} \rceil -2$ many queries. According to the information theory lower bound we need $\log \binom{k+1}{d}$ queries to find the at most $d$ defective elements among the remaining $k+1$ elements.
\end{proof}

\section*{Acknowledgments}
The authors would like to thank G.O.H. Katona for guidance and helpful discussions.  Additionally, the authors would like to acknowledge the hospitality of the organizers of the fifth Eml\'ekt\'abla workshop, during which the majority of this research was conducted.

\bibliographystyle{plain}

\bibliography{d-sepbib}

\end{document}